\documentclass[onefignum,onetabnum]{siamart190516}

\usepackage{lipsum}
\usepackage{amsfonts}
\usepackage{graphicx}
\usepackage{epstopdf}
\usepackage{algorithmic}
\ifpdf
  \DeclareGraphicsExtensions{.eps,.pdf,.png,.jpg}
\else
  \DeclareGraphicsExtensions{.eps}
\fi

\usepackage{amsopn}
\usepackage{tikz}
\usetikzlibrary{scopes}
\usetikzlibrary{intersections}
\usepackage{times,amsmath,amsbsy,amssymb,amscd,mathrsfs}
\usepackage{graphicx,subfigure,epstopdf,wrapfig,chemarrow}

\newtheorem{remark}[theorem]{Remark}
\newtheorem{example}[theorem]{example}

\usepackage[numbered]{mcode}
\usepackage{slashbox}
\usepackage{booktabs}
\renewcommand\arraystretch{1.5}
\usepackage{times,amsmath,amsbsy,amssymb,amscd,mathrsfs}

\newcounter{mnote}
\let\oldmarginpar\marginpar
\renewcommand\marginpar[1]{\-\oldmarginpar[\raggedleft\footnotesize #1]%
	{\raggedright\footnotesize #1}}

\newsiamremark{assum}{Assumption} 
\newsiamremark{rem}{Remark} 
\newsiamthm{lem}{Lemma} 
\newsiamthm{prop}{Proposition} 
\newsiamthm{thm}{Theorem} 
\newsiamthm{coro}{Corollary} 
\newsiamthm{Def}{Definition} 

\begin{document}
		\headers{Randomized Fast Subspace Descent Methods}
	{L. Chen, X. Hu, and H. Wu}
	\title{ Randomized Fast Subspace Descent Methods\thanks{The work of X.~Hu is partially supported by the National Science Foundation under grant DMS-1812503 and CCF-1934553. L. Chen and H. Wu are partially supported by the National Science Foundation under grant DMS-1913080.}}
\author{Long Chen\thanks{Department of Mathematics, University of California at Irvine, Irvine, CA 92697, USA\,(\email{chenlong@math.uci.edu}).}
\and Xiaozhe Hu \thanks{Department of Mathematics, Tuffs University, Medford, MA 02155, USA \, (\email{Xiaozhe.Hu@tufts.edu})}
	\and Huiwen Wu \thanks{Department of Mathematics, University of California at Irvine, Irvine, CA 92697, USA\,(\email{huiwenw@uci.edu}).}}

	\maketitle

\begin{abstract}
Randomized Fast Subspace Descent (RFASD) Methods are developed and analyzed for smooth and non-constraint convex optimization problems. The efficiency of the method relies on a space decomposition which is stable in $A$-norm, and meanwhile the condition number $\kappa_A$ measured in $A$-norm is small. At each iteration, the subspace is chosen randomly either uniformly or by a probability proportional to the local Lipschitz constants. Then in each chosen subspace, a preconditioned gradient descent method is applied. RFASD converges sublinearly for convex functions and linearly for strongly convex functions.  Comparing with the randomized block coordinate descent methods, the convergence of RFASD is faster provided $\kappa_A$ is small and the subspace decomposition is $A$-stable. This improvement is supported by considering a multilevel space decomposition for Nesterov's `worst' problem. 
\end{abstract}
	
	

	\tableofcontents

\section{Introduction}
We consider the non-constraint convex minimization problem 
\begin{equation}\label{eq:min}
\min_{x \in \mathcal{V}} f(x),
\end{equation}
where $f$ is a smooth and convex function and its derivative is Lipschitz continuous with constant $L$ and $\mathcal{V}$ is a Hilbert space. In practice, $\mathcal{V} = \mathbb{R}^N$ but might be assigned with an inner product other than the standard $l^2$ inner product of $\mathbb{R}^N$. Solving minimization problem \eqref{eq:min} is a central task with wide applications in fields of scientific computing, machine learning, and data science, etc.

Due to the eruption of data and the stochasticity of the real world, randomness is introduced to make algorithms more robust and computational affordable. In the following, we will restrict ourselves to randomized algorithms related to the coordinate descent (CD) method~\cite{ortegaIterativeSolutionNonlinear2000,bertsekasParallelDistributedComputation1989,luoConvergenceCoordinateDescent1992} and its block variant, i.e., the block CD (BCD) method~\cite{ortegaIterativeSolutionNonlinear2000,bertsekasParallelDistributedComputation1989,bertsekasNonlinearProgramming1997,tsengConvergenceBlockCoordinate2001,BeckTetruashvili2013} and propose a new algorithm generalizing the randomized CD (RCD) and randomized BCD (RBCD) methods. 

In~\cite{nesterov2012efficiency}, Nesterov studied a RBCD method for huge-scale optimization problems.  Assuming the gradient of the objective function $f$ is coordinate-wise Lipschitz continuous with constants $L_i$, at each step, the block coordinates is chosen randomly with probability $p_i = L_i^{\alpha} \cdot [\sum_{j=1}^n L_j^{\alpha}]^{-1}, ~ \alpha \in \mathbb{R}$ and an optimal block coordinate step with step size $1/L_i$ is employed. Note that, when $\alpha = 0$, the probability $p_i$ is uniformly distributed. When $\alpha = 1$, it is proportional to $L_i$.  It is shown that such an RBCD method converges linearly for a strongly convex function $f$ and sublinearly for the convex case. Later, in~\cite{BeckTetruashvili2013}, the cyclic version of the BCD method was studied, namely, each iteration consists of performing a gradient projection step with respect to a certain block taken in a cyclic order.  Global sublinear convergence rate was established for convex problems and, when the objective function is strongly convex, a linear convergence rate can be proved.   
%
More recently, Wright ~\cite{Wright2015} studied a simple version of RCD that updates one coordinate at each time with uniformly chosen index. It was pointed out that when applying to linear system $Ax = b$ using least-sqaures formulation, such a RCD is exactly a randomized Kaczmarz method~\cite{strohmer2009randomized,leventhalRandomizedMethodsLinear2010}. Similarly, it was shown that RCD convergences sublinearly for convex functions and linearly for strongly convex functions. 
In~\cite{lu2017randomized}, Lu developed a randomized block proximal damped newton (RBPDN) method. For solving smooth convex minimization problem, RBPDN uses Newton's method in each block. Comparing with the Newton's method, the computational complexity of RBPDN is reduced since the Newton's step is performed locally on each block.  There is a trade-off between convergence rate and computational complexity.  If the dimension of the blocks is too small, i.e. $\mathcal O(1)$, the Hessian on each block might lose lots of information, which might lead to slow convergence. While if the block's dimension is large, for example, $N/2$, then the  computation of Hessian inverse on each block might still be expensive.  

Those existing RCD and RBCD methods can achieve acceleration comparing with standard gradient descent (GD) methods, especially for large-scale problems.  However, the convergence of RCD and RBCD becomes quite slow when the problem is ill-conditioned.  It is well-known that, preconditioning techniques can be used to improve the conditioning of an optimization problem, see~\cite{polyakIntroductionOptimizationOptimization1987,ortegaIterativeSolutionNonlinear2000}.  While preconditioning techniques can be motivated in different ways, one approach is to look at the problem~\eqref{eq:min} in $\mathcal{V}$ endowed by an inner product induced by the preconditioner.  Roughly speaking, assuming the preconditioner $A$ is symmetric and positive-definite, we consider the Lipschitz continuity and convexity using the $A$-inner product and $A$-norm.  A good preconditioner means that the condition number measured using $A$-norm is relatively small and therefore, the convergence of preconditioned GD (PGD) can be accelerated.  The price to pay is the cost of the action of $A^{-1}$, which might be prohibitive for large-size problem.  Moreover, it is also difficult to use the preconditioner in the RCD and RBCD methods due to the fact that the coordinate-wise decomposition is essentially based on the $l^2$-norm.  



One main idea of the proposed algorithm is to generalize the coordinate-wise decomposition to subspace decomposition which is more suitable for the $A$-norm.  This idea itself is not new.  For example, the well-known multigrid method~\cite{stuben2001review}, which is one of the most efficient methods for solving the elliptic problem, can be derived fom subspace correction methods based on a multilevel space decomposition~\cite{Xu:1992Iterative}. Its randomized version has been considered in~\cite{huRandomizedFaulttolerantMethod2019a}.  Recently, in~\cite{Chen;Hu;Wise:2018Convergence}, we have developed fast subspace descent (FASD) methods by borrowing the subspace decomposition idea of multigrid methods for solving the optimization problems.  In this paper, we provide a randomized version of FASD and abbreviated as RFASD. 

A key feature of FASD and RFASD is a subspace decomposition $\mathcal V = \mathcal V_1 + \mathcal V_2 + \cdots \mathcal V_J, $ with $\mathcal V_i \subset \mathcal V,$ $i = 1, \cdots, J$.  Note here we do not require the space decomposition to be a direct sum nor be orthogonal. Indeed, 
the overlapping/redundancy between the subspaces is crucial to speed up the convergence if the space decomposition is stable in the $\| \cdot \|_A$ norm as follows,
 \begin{itemize}
	\item (SD) Stable decomposition: there exists a constant $C_{_A} > 0$, such that
\begin{equation}\label{eq:introSD}
	\forall v \in \mathcal{V}, \quad v = \sum _{i=1}^J v_i, \quad \text{ and }\quad \sum_{i=1}^J \|v_i\|_{A}^2 \leq C_{_A}\|v\|_{A}^2.
\end{equation}
\end{itemize}
With such a subspace decomposition, the proposed RFASD method is similar with the RBCD method. At each iteration, RFASD randomly chooses a subspace according to certain sampling distribution, computes a search direction in the subspace, and then update the iterator with an appropriate step size. 



Based on standard assumptions, we first prove a sufficient decay inequality. Then coupled with the standard upper bound of the optimality gap, we are able to show RFASD converges sublinearly for convex functions and linearly if the objective function is strongly convex.  More importantly, the convergence rate of RFASD depends on the condition number measured in $A$-norm and there is no need of inverting $A$ directly, only local solves on each subspace is sufficient.  Using strongly convex case as an example, we show that the convergence rate is 
\begin{equation*}
1 - \frac{1}{J}\frac{1}{C_{_A}}\frac{1}{\kappa_{_A}},
\end{equation*}
where $\kappa_{_A}$ is the condition number of $f$ measured in the $A$-norm, which could be much smaller than the condition number of $f$ measures in $l^2$-norm.  Here $J$ is the number of subspaces and $C_{_A}$ measures the stability of the space decomposition in $A$-norm; see \eqref{eq:introSD}.  Therefore, if we choose a proper preconditioner $A$ such that $\kappa_{_A} = \mathcal{O}(1)$ and a stable subspace decomposition such that $C_{_A} = \mathcal{O}(1)$,  after $J$ iterations, we have 
\begin{equation*}
\left( 1 - \frac{1}{J}\frac{1}{C_{_A}}\frac{1}{\kappa_{_A}} \right)^J \leq \exp(-\frac{1}{C_{_A}\kappa_{_A}}) = \exp(-\mathcal{O}(1)).
\end{equation*}
This indicates an exponential decay rate that is independent of the size of the optimization problem, which shows the potential of the proposed RFASD method for solving large-scale optimization problems.  In summary, based on a stable subspace decomposition, we can achieve the preconditioning effect by only solving smaller size problems on each subspace, which reduces the computational complexity.

The paper is organized as follows. In Section~\ref{sec:RFASD}, we set up the minimization problem $\min_{x \in \mathcal{V}} f(x)$ with proper assumptions on $f$ and $\nabla f$. Then we propose the RFASD algorithm. In Section~\ref{sec:convergence}, based on the stable decomposition assumption,  convergence analysis for convex functions and strongly convex function are derived.  In Section~\ref{sec:examples}, we give some examples and comparisons between several methods within the framework of RFASD.  Numerical experiments results are provided in Section~\ref{sec:numerics} to confirm the theories.  Finally, we provide some conclusions in Section~\ref{sec:conclusions}.


\section{Fast Subspace Descent Methods} \label{sec:RFASD}
In this section, we introduce the basic setting of the optimization problem we consider as well as basic definitions, notation, and assumptions. Then we propose the fast subspace descent method based on proper subspace decomposition. 

\subsection{Problem Setting}
We consider the minimization problem \eqref{eq:min}. The Hilbert space $\mathcal V$ is a vector space equipped with an inner product 
$(\cdot, \cdot)_{A}$ and the norm induced is denoted by $
\| x \|_{A} = (x, x )_{A}^{1/2}$. Although our discussion might be valid in general Hilbert spaces, we restrict ourself to the finite dimensional space and without of loss generality we take $\mathcal V = \mathbb R^N$. In this case, the standard $l^2$ dot product in $\mathbb R^N$ corresponds to $A = I$ and is denoted by
\begin{equation}\label{def:dot_prod}
(x, y) = x \cdot y : = \sum_{i=1}^N x_i y_i\quad \forall x,y \in \mathcal{V}.
\end{equation}
The inner produce $(\cdot, \cdot)_A$ is induced by a given symmetric positive definite (SPD) matrix $A \in \mathbb{R}^{N \times N}$ and defined as follows,
\begin{equation}
(x, y)_A := (A x, y), \quad  \forall ~x,y \in \mathcal{V}. 
\end{equation}
Let $\mathcal V' := \mathcal{L} (\mathcal V, \mathbb R)$ be the linear space of all linear and continuous mappings $\mathcal V \rightarrow \mathbb R,$ which is called the dual space of $\mathcal V$. The dual norm w.r.t the $A$-norm is defined as: for $f\in \mathcal V'$
\begin{equation}\label{def:norm_Vp}
\| f \|_{\mathcal V' } = \sup_{\| x \|_{A} \leq 1 } \langle f, x \rangle,
\end{equation}
where $\langle \cdot, \cdot \rangle$ denotes the standard duality pair between $\mathcal{V}$ and $\mathcal{V}'$.  By Riesz representation theorem, $f$ can be also treat as a vector and, it is straightforward to verify that 
$$
\| f \|_{\mathcal V' } = \| f \|_{A^{-1}} := \langle A^{-1} f, f \rangle^{1/2}.
$$

%
%
%

Next, we introduce a decomposition of the space $\mathcal{V}$, i.e., 
\begin{equation} \label{eqn:space-decomp}
\mathcal{V} = \mathcal{V}_1 + \mathcal{V}_2 + \cdots + \mathcal{V}_J, \quad \mathcal{V}_i \subset \mathcal{V}, \quad i = 1,\cdots, J.
\end{equation}
Again we emphasize that the space decomposition is not necessarily a direct sum nor be orthogonal. For each subspace $\mathcal V_i$, we assign an inner product induced by a symmetric and positive definite matrix $A_i: \mathcal V_i\to \mathcal V_i$. The product space $$\widetilde V = \mathcal V_1\times \mathcal V_2 \times \cdots \times \mathcal V_J$$ is assigned with the product topology: for $\tilde v = (v_1, v_2, \ldots, v_J)^{\intercal }\in \widetilde V$
$$
\| \tilde v \|_{\tilde A} := \left (\sum_{i=1}^J \| v_i \|_{A_i}^2\right )^{1/2}. 
$$
In matrix form, $\tilde A = {\rm diag}(A_1, A_2, \ldots, A_J)$ is a block diagonal matrix defined on $\widetilde V$. 

We shall make the following assumptions on the objective function:
\begin{itemize}
\item (LCi) The gradient of $f$ is Lipschitz continuous restricted in each subspace with Lipschitz constant $L_i$, i.e., 
\begin{equation*}
\| \nabla f(x+v_i) - \nabla f(x) \|_{A^{-1}} \leq L_{A,i} \| v_i \|_{A_i}, \quad \forall \, v_i \in \mathcal{V}_i.
\end{equation*}

\item (SC) $f$ is strongly convex with strong convexity constant $\mu \geq 0$, i.e.,
\begin{equation*}
\langle \nabla f(x) - \nabla f(y), x - y \rangle \geq \mu_{_A} \| x - y \|_{A}^2, \quad \forall \, x, \, y \in \mathcal{V}.
\end{equation*}

\end{itemize}




Let $I_i: \mathcal{V}_i \mapsto \mathcal{V}$ be the natural inclusion and let $R_i = I_i^{\intercal}: \mathcal V' \mapsto \mathcal V_i'$. In the terminology of multigrid methods, $I_i$ corresponds to the prolongation operator and $R_i$ is the restriction.

\subsection{Randomized Fast Subspace Descent Methods}
Now, we propose the randomized fast subspace descent (RFASD) algorithm. 

\begin{algorithm}[H]
\caption{Randomized Fast Subspace Descent Method}\label{alg:RFASD}
\begin{algorithmic}[1]
\STATE Choose $x^0$ and $k \gets 0$	
\FOR{$k = 0,1, \cdots $}
\STATE Choose an index of subspace $i_k$ from $\{1,\cdots, J\}$ with the sampling probability:
\begin{equation}\label{eq:nonuniform}
i_k = i \quad \text{ with probability } p_i = \frac{1}{J} \frac{L_{A,i}}{\bar L_{_A}}, \quad i = 1,2, \ldots, J,
\end{equation}
where $ \displaystyle \bar L_{_A} = \frac{1}{J} \sum_{i=1}^J L_{A,i}$ is the averaged Lipschitz constant.

\smallskip
\STATE Compute subspace search direction $\displaystyle s_{i_k} = - I_{i_k} A_{i_k}^{-1}R_{i_k}\nabla f(x^k)$

\smallskip
\STATE Choose the step size $\displaystyle \alpha_k = \frac{1}{L_{A,i_k}}$ and update by the subspace correction:
$$
x^{k+1}: = x^k + \alpha_k s_{i_k}. 
$$
\ENDFOR
\end{algorithmic}
\end{algorithm}

%

The non-uniform sampling distribution and the step size $\alpha_k = 1/L_{A,i_k}$ requires a priori knowledge of $L_{A,i}$. A conservative plan is to use one upper bound for all subspaces. For example, when the gradient of $f$ is Lipschitz continuous with Lipschitz constant $L_{_A}$, i.e., 
\begin{equation*}
\| \nabla f(x) - \nabla f(y) \|_{A^{-1}} \leq L_{_A} \| x - y \|_{A}, \quad \forall \, x, \, y \in \mathcal{V}.
\end{equation*}
We can set $L_{A,i} = L_{_A}$ for all $i$ and consequently we pick the subspace uniformly and use a uniform step size $1/L_{_A}$. 

There is a balance between the number of subspaces and the complexity of the subspace solvers. For example, we can choose $J=1$ and thus $C_{_A}=1$. But then we need to compute $A^{-1}$ which may cost $\mathcal O(N^p)$ for $p > 1$ which is not practical for large-size problems (e.g., using the standard Gauss elimination to compute $A^{-1}$ leads to $p=3$). On the other extreme, we can chose a multilevel decomposition with  $J = \mathcal O(N\log N)$ and $n_i = \mathcal O(1)$. Then the cost to compute $A_{i}^{-1}$ is $\mathcal O(1)$.  


One important question is what is a good choice of an $A$-norm? Any good preconditioner for the objective function is a candidate. For example, when $\nabla^2 f$ exists, $A = \nabla^2 f(x^k)$ or its approximation is a good choice since this inherits advantages of the Newton's method or quasi-Newton methods.   

Another important question is how to get a stable decomposition based on a given SPD matrix $A$? There is no satisfactory and universal answer to this question. One can always start from a block coordinate decomposition. When $A = \nabla^2 f(x^k)$, this leads to the block Newton method considered in~\cite{lu2017randomized}. We can then merge small blocks to form a larger one in a multilevel fashion and algebraic multigrid methods~\cite{stuben2001review} can be used in this process to provide a coarsening of the graph defined by the Hessian. In general, efficient and effective space decomposition will be problem dependent. We shall provide an example later on.

\subsection{Randomized Full Approximation Storage Scheme} 
The full approximation storage (FAS) scheme, in the deterministic setting, was proposed in~\cite{Brandt1977} and is a multigrid method for solving nonlinear equations. Several FAS-like algorithms for solving optimization problems have been considered in the literature \cite{Gelman;Mandel:1990multilevel,Lewis;Nash:2005problems,Nash.S2000,Gratton;Sartenaer;Toint:2008Recursive}, including those that are line search-based recursive or trust region-based recursive algorithms

Based on RFASD, we shall propose a randomized FAS (RFAS) algorithm.  We first recall FAS in the optimization setting as discussed in~\cite{Chen;Hu;Wise:2018Convergence}.  Given a space decomposition $
\mathcal{V} = \mathcal{V}_1 + \mathcal{V}_2 + \cdots + \mathcal{V}_J, \ \mathcal{V}_i \subset \mathcal{V}, \ i = 1,\cdots, J.$ Let $Q_i: \mathcal{V} \mapsto \mathcal{V}_i$ be a projection operator and, ideally, $Q_i v$ should provide a good approximation of $v \in \mathcal{V}$ in the subspace $\mathcal{V}_i$. In addition,  $f_i: \mathcal{V}_i \mapsto \mathbb{R}$ is a local objective function. $f_i$ can be the original $f$. Then it coincides with the multilevel optimization methods established by Tai and Xu~\cite{Tai;Xu:2001Global}; see Remark 4.2 in~\cite{Chen;Hu;Wise:2018Convergence}. 
%
Given the current approximation $x^k$, in FAS, the search direction $s_i$, $i=1,2,\cdots, J$, is computed by the following steps:

\begin{algorithm}[H]
	\caption{Compute search direction $s_i$ for FAS} \label{alg:FAS-search}
\textbf{Input:} current approximation $x^k$ and index $i$

\textbf{Output:} search direction $s_i$
\begin{algorithmic}[1]
 \STATE Compute the so-called $\tau$-correction: $\tau_i = \nabla f_i(Q_i x^k) - R_i \nabla f(x^k)$
\smallskip
 \STATE Solve the $\tau$-perturbed problem: $
 \langle \nabla f_i(\eta_i), v_i \rangle = \langle \tau_i, v_i \rangle, \ \forall \, v_i \in \mathcal{V}_i.$
\smallskip
 \STATE Compute the search direction: $s_{i}:= \eta_i - Q_i x^k$. 
\end{algorithmic}
\end{algorithm}

Replacing Step 4 in RFASD (Algorithm~\ref{alg:RFASD}), we obtain the RFAS as shown in Algorithm~\ref{alg:RFAS}.
\begin{algorithm}[H]
	\caption{Randomized Full Approximation Storage Scheme} \label{alg:RFAS}
	\begin{algorithmic}[1]
		\STATE Choose $x^0$ and $k \gets 0$	
		\FOR{$k = 0,1, \cdots $}
		\STATE Choose an index $i_k$ from $\{1,\cdots, J\}$ with the sampling probability given as~\eqref{eq:nonuniform}.
		\smallskip
		\STATE Compute subspace search direction $s_{i_k}$ using Algorithm~\ref{alg:FAS-search} with inputs $x^k$ and $i_k$. 
		
		\smallskip
		\STATE Update by the subspace correction:
		$
		\displaystyle
		x^{k+1}: = x^k + \frac{1}{L_{A, i_k}} s_{i_k}. 
		$
		\ENDFOR
	\end{algorithmic}
\end{algorithm}

Next we show that if we choose $f_i$ in certain way, RFAS becomes a special case of RFASD.  Given the SPD matrices $A_i$, which induce the inner products on subspaces $\mathcal{V}_i$, $i= 1,2,\cdots, J$, we define the following quadratic local objective functions,
\begin{equation*}
f_i(w) = \frac{1}{2} \| w \|^2_{A_i}, \quad \forall \, w\in \mathcal V_i, \quad i = 1, 2, \cdots, J.
\end{equation*}
From Algortihm~\ref{alg:FAS-search}, it is easy to see that $s_i = -I_i A_i^{-1}R_i \nabla f(x^k)$.  Therefore, in this case, RFAS agrees with RFASD. Note that in this setting $A_i$ may not be the Galerkin projection of $A$, i.e. $A_i \neq R_i AI_i$, which is different from RPSD.  Nevertheless, the convergence analysis (Theorem~\ref{thm:str-con-linear} and \ref{thm:conv-sublinear}) can be still applied but the constants $L_{A,i}$ and $C_{_A}$ depends on the choices of $A_i$.  


Consider a slighly more general case that the local objective function $f_i$ is in $\mathcal C^2$, then by mean value theorem, from Algorithm~\ref{alg:FAS-search}, we can write the equation of $s_i$ as
$$
\langle \nabla^2f_i (\xi_i) s_i, v_i \rangle = - \langle R_i \nabla f(x^k), v_i \rangle, \quad \forall \, v_i \in \mathcal{V}_i,
$$ 
which implies $A_i = R_i \nabla^2f_i (\xi_i) I_i$ and, again, RFAS is a special case of RFASD.  Of course, this choice of $A_i$ is impractical because we do not know $\xi_i$ in general.  One practical choice might be $A_i = R_i\nabla^2f (x^k) I_i$, which is the Galerkin projection of the Hessian matrix of original $f$. In this case, RFAS becomes block Newton's method. The constants $L_{A,i}$ and $C_{_A}$ are thus changed as $A_i$ depends on $x^k$ and is different at each iteration. 

\section{Convergence Analysis} \label{sec:convergence}
In this section, we shall present a convergence analysis of RFASD.  We first discuss the stability of a space decomposition and then prove the crucial sufficient decay property. Then we obtain a linear or sublinear convergence rate by different upper bounds of the optimality gap. 

\subsection{Stable decomposition}
We first introduce the mapping $\Pi: \widetilde V\to \mathcal V$ as follow, 
$$
\Pi \tilde v = \sum_{i=1}^J v_i, \quad \text{for} \ \tilde v = (v_1, v_2, \ldots, v_J)^{\intercal }\in \widetilde V.
$$
We can write $\Pi = (I_1, I_2, \ldots, I_J)$ and $\Pi^{\intercal } = (R_1, R_2, \ldots, R_J)^{\intercal}$ in terms of prolongation and restriction operators. 

The space decomposition $\mathcal V = \sum_i \mathcal V_i$ implies that $\Pi$ is surjective. Since for finite dimensional spaces, all norms are equivalent, $\Pi$ is a linear and continuous operator and, thus, by the open mapping theorem, there exists a continuous right inverse of $\Pi$. Namely there exists a constant $C_{_A} > 0$, such that for any $v \in \mathcal{V}$,  there exists a decomposition $v = \sum _{i=1}^J v_i, $ with $v_i\in \mathcal V_i$ for $i=1,\ldots, J$, and
\begin{equation}\label{eq:CA}
\sum_{i=1}^J \|v_i\|_{A_i}^2 \leq C_{_A}\|v\|_A^2.
\end{equation}
The constant $C_{_A}$ measures the stability of the space decomposition. When the decomposition is orthogonal, $C_{_A} = 1$. As the adjoint, the operator $\Pi^{\intercal}: \mathcal V'\to \widetilde V'$ is injective and bounded below. The following result is essentially from the fact that $\Pi$ and $\Pi^{\intercal}$ has the same minimum singular value $1/C_{_A}$.

\begin{lemma}
For a given $g\in \mathcal V'$, let $g_i = R_i g, s_i = - A_i^{-1} g_i$ for $i=1,2,\ldots, J$. Then 
\begin{equation}\label{eq:gs}
-\langle g, \sum_{i=1}^J s_i \rangle = \sum_{i=1}^J\| g_i \|_{A_i^{-1}} = \sum_{i=1}^J \|s_i\|_{A_i}^2
\end{equation}
and
\begin{equation}\label{eq:gstable}
\| g\|_{A^{-1}}^2 \leq C_{_A}\sum_{i=1}^J \|s_i\|_{A_i}^2 = \sum_{i=1}^J \|s_i\|_{A_i}^2,
\end{equation}
where $C_{_A}$ is the constant in \eqref{eq:CA}.
\end{lemma}
\begin{proof}
The first identity is an easy consequence of definitions as: for $i=1,2,\ldots, J$
 $$
-\langle g, s_i \rangle = \langle g, I_i A_i^{-1} R_i g \rangle =  \| g_i \|_{A_i^{-1}} = \| A_i^{-1} g_i \|_{A_i}^2 = \|s_i\|_{A_i}^2.
 $$
 
We now prove \eqref{eq:gstable}.  
For a given $g\in \mathcal V'$, let $g_i = R_i g \in \mathcal V_i'$. 
For any $w\in \mathcal V$, we chose a stable decomposition $w = \sum_{i=1}^J w_i, w_i \in \mathcal V_i$. Then
\begin{align*}
 \langle g, w \rangle  &= \sum_{i=1}^J \langle g, w_i \rangle  = \sum_{i=1}^J \langle g_i, w_i \rangle \leq \left (\sum_{i=1}^J \|g_i\|_{A_i^{-1}}^2 \right )^{1/2} \left (\sum_{i=1}^J \|w_i \|_{A_i}^2 \right )^{1/2}\\
						  & \leq C_{_A}^{1/2} \left (\sum_{i=1}^J \| g_i \|_{A_i^{-1}}^2 \right )^{1/2} \| w\|_{A}.
\end{align*} 
Thus,
$$
 \| g \|_{A^{-1}}^2  = \left ( \sup_{w\in \mathcal V} \frac{\langle g, w \rangle}{\|w \|_{A}} \right )^2 \leq C_{_A} \sum_{i=1}^J \| g_i \|_{A_i^{-1}}^2 = \sum_{i=1}^J \|s_i\|_{A_i}^2,
$$
which completes the proof.
\end{proof}


\subsection{Sufficient decay}
We shall prove a sufficient decay property for the function value. Note that we do not assume $f$ is convex but only Lipschitz continuous in each subspace.

\begin{lemma}
 Suppose the objective function $f$ and space decomposition $\mathcal V = \sum_{i=1}^J \mathcal V_i$ satisfy (LCi). Let $\{ x^k\}$ be the sequence generated by Algorithm \ref{alg:RFASD}. Then for all $k>0$, we have
\begin{equation}\label{eq:decay}
\mathbb E[f(x^{k+1})] - f(x^k) \leq -\frac{1}{2\bar L_{_A} C_{_A} J}\| \nabla f(x^k) \|_{A^{-1}}^2.
\end{equation}
\end{lemma}
\begin{proof}
 By the Lipschitz continuity (LCi) and the choice of the step size, we have
\begin{align*}
f(x^{k+1}) &\leq f(x^k) + \alpha_k \langle \nabla f(x^k), s_{i_k}\rangle  + \frac{L_{A,i_k}}{2} \alpha_k^2  \| s_{i_k} \|_{A_i}^2\\
&= f(x^k) + \frac{1}{L_{A,i_k}} \langle \nabla f(x^k), s_{i_k}\rangle  + \frac{1}{2 L_{A,i_k}}  \| s_{i_k} \|_{A_i}^2.
\end{align*}
Take expectation of $i_k$ conditioned by $x^k$ with probability $p_i = \frac{1}{J}L_{A,i}/ \bar L_{_A},\,  j = 1, \cdots, J.$
\begin{align*}
\mathbb{E}\left [ f(x^{k+1}) \right) - f(x^k) & \leq \frac{1}{J\bar L_{_A}} \langle \nabla f(x^k), \sum_{i=1}^J s_i \rangle + \frac{1 }{2J \bar L_{_A}} \sum_{i=1}^J \|s_i\|_{A_i}^2  \\
 & \leq -  \frac{1 }{2J \bar L_{_A}} \sum_{i=1}^J \|s_i\|^2_{A_i} \quad \text{ by \eqref{eq:gs}} \\
& \leq -\frac{1}{2 J \bar L_{_A}C_{_A}} \| \nabla f(x^k) \|_{A^{-1}}^2 \quad  \text{ by \eqref{eq:gstable}}. 
\end{align*}
\end{proof}
As we will show in the next two subsections, based on the above sufficient decay property~\eqref{eq:decay}, together with a proper upper bound of the optimality gap, linear or sub-linear convergent rate can be obtained for the strongly convex and convex case, respectively.

\subsection{Linear convergence for strongly convex functions}
To derive the linear convergence for the strongly convex case, we shall use the following upper bound of the optimality gap. 
	\begin{lemma}[Theorem 2.1.10 in~\cite{nesterov2013introductory}]
	\label{lm:dE-energy-V}
Suppose that $f$ satisfies assumption (SC) with constant $\mu_{_A}>0$ and $x^* \in \mathcal{V}$ is the minimizer of $f$; then for all $x \in \mathcal{V}$, 
	\begin{equation}
	\label{eq:upper} 
f(x) - f(x^*)  \leq \frac{1}{2\mu_{_A}} \| \nabla f(x) \|_{A^{-1}}^2.
	\end{equation}
	\end{lemma}

Now we are ready to show the linear convergence of RFASD when the objective function $f$ is strongly convex and the result is summarized in the following theorem. 
\begin{theorem}\label{thm:str-con-linear}
 Suppose the objective function and space decomposition satisfy (LCi) and (SC) with $\mu_{_A} > 0$. Let $\{ x^k\}$ be the sequence generated by Algorithm \ref{alg:RFASD}. Then for all $k>0$, we have the linear contraction
\begin{equation}\label{eq:str-conv-linear}
\mathbb{E}\left [ f(x^{k+1}) \right ] - f(x^*) \leq \left( 1 - \frac{1}{J}\frac{\mu_{_A}}{\bar L_{_A}}\frac{1}{C_{_A}} \right) \left( \mathbb{E} \left[ f(x^k) \right] - f(x^*) \right).
\end{equation}
\end{theorem}
\begin{proof}
The left hand side of \eqref{eq:decay}  can be rewritten as $\mathbb E[f(x^{k+1})] - f(x^*)- ( f(x^k) - f(x^*))$. Combining \eqref{eq:decay} and \eqref{eq:upper}, rearranging the terms, and taking expectation with respect to $x^k$,  we get the desired result.
\end{proof}

\subsection{Sublinear convergence for convex functions}
Next, we give the convergence result for convex but not strongly convex objective functions, i.e. $\mu = 0$ in (SC), based on the following bounded level set assumption. 
\begin{itemize}
\item (BL) Bounded level set:
$f$ is convex and attains its minimum value $f^*$ on a set $S$. There is a finite constant $R_0$ such that the level set for $f$ defined by $x^0$ is bounded, that is, 
\begin{equation}
\max_{x^* \in S} \max_{x} \{ \| x - x^* \|_{A}: f(x) \leq f(x^0) \} \leq R_0. 
\end{equation} 
\end{itemize}

\begin{lemma}
 	\label{lm:convex-upper}
Suppose the objective function $f$ satisfies (LC) and (BL). Then for all $x \in \mathcal{V}$ and $f(x)\leq f(x^0)$, 
	\begin{equation}
	\label{eq:convex-upper} 
f(x) - f(x^*)  \leq R_0 \| \nabla f(x) \|_{A^{-1}}.
	\end{equation}
\end{lemma}
\begin{proof} By convexity and (BL), for $x \in \mathcal{V}$ and $f(x)\leq f(x^0)$,
 \begin{equation*}
f(x) - f^* \leq \langle \nabla f(x^k), x - x^* \rangle \leq \| \nabla f(x) \|_{A^{-1}} \| x - x^* \|_{A} \leq R_0 \| \nabla f(x) \|_{A^{-1}}.
\end{equation*}
\end{proof}

We still use the same step size and show that RFASD converges sublinearly for convex objective function $f$.

\begin{theorem}\label{thm:conv-sublinear}
 Suppose the objective function and space decomposition satisfy (LCi), (BL) and (SC) with $\mu_{_A} = 0$. Let $\{ x^k\}$ be the sequence generated by Algorithm \ref{alg:RFASD}. Then for all $k>0$, we have
\begin{equation}\label{ine:conv-sublinear}
\mathbb{E}\left [ f(x^k) \right ] - f^* \leq \frac{d_0}{1 + d_0\, c\, k} \leq \frac{2R_0^2 J \bar{L}_{_A} C_{_A}}{k},
\end{equation}
where $d_0 = f(x^0) -f^*$ and $c = 1/(2R_0^2J \bar L_{_A}C_{_A}).$

\end{theorem}
\begin{proof}
Note that \eqref{eq:decay} can be written as
\begin{align*}
\mathbb{E} \left[f(x^{k+1}) \right] - f(x^*) - \left( f(x^k) - f(x^*) \right) &\leq  - \frac{1}{2 \bar L_{_A}J}\frac{1}{C_{_A}}\| \nabla f(x^k) \|_{A^{-1}}^2 \\
&\leq -c  \left( f(x^k) - f(x^*) \right)^2
\end{align*}
where $c = 1/(2R_0^2J \bar L_{_A}C_{_A}).$ and the last inequality is derived based on~\eqref{eq:convex-upper} with $x=x^k$.  Now denoting $d_{k} = \mathbb{E}\left[ f(x^k)\right] - f^*$ and taking expectation with respect to $x^k$, we have
\begin{equation*}
d_{k+1} - d_k \leq -c d_{k}^2 \leq 0.
\end{equation*}
Based on the above inequality, we obtain
\begin{equation*}
\frac{1}{d_{k+1}} - \frac{1}{d_{k}} = \frac{d_{k} - d_{k+1}}{d_{k} d_{k+1}} \geq \frac{d_{k} - d_{k+1}}{\left(d_{k}\right)^2} \geq c.
\end{equation*}
Recursively applying this inequality, we obtain 
\begin{equation} \label{ine:d_k}
\frac{1}{d_{k}} \geq \frac{1}{d_0} + c k. 
\end{equation}
which implies~\eqref{ine:conv-sublinear}. 
\end{proof}

\begin{remark}
The parameters $R_0, J, \bar L_{_A}$, and $C_{_A}$ can be dynamically changing, i.e., as a function of $k$. For example, we can use $R_k:=\max_{x^* \in S} \max_{x} \{ \| x - x^* \|: f(x) \leq f(x^k) \}$, which is smaller than $R_0$. The space decomposition and local Lipschitz constants could also be improved during the iterations. In these cases, we use $c_k$ to denote the constant and the last inequality~\eqref{ine:d_k} holds with the second term $ck$ on the right-hand-side being replaced by $\sum_{i=1}^k c_i.$
\end{remark}

\subsection{Complexity}
Based on the convergence results Theorem~\ref{thm:str-con-linear} and~\ref{thm:conv-sublinear}, we can estimate the computational complexity of the proposed RFASD method and compare with the GD, PGD, RCD, and RBCD methods.  As usual, for a prescribed error $\epsilon > 0$, we first estimate how many iterations are needed to reach the tolerance and then estimate the overall computational complexity based on the cost of each iteration.  

For gradient-based methods, the main cost per iteration is the evaluation of gradient $\nabla f(x^k)$.  In general, it may take $\mathcal{O}(N^2)$ operations.  In certain cases, the cost could be reduced.  For example, when $\nabla f(x^k)$ is \emph{sparse}, e.g., computing one coordinate component of $\nabla f(x^k)$ only needs $\mathcal{O}(1)$ operations, then the computing $\nabla f(x^k)$ takes $\mathcal{O}(N)$ operations.  Another example is to use advanced techniques, such as fast multipole method~\cite{rokhlinRapidSolutionIntegral1985,greengardFastAlgorithmParticle1987}, to compute $\nabla f(x^k)$, then the cost could be reduced to $\mathcal{O}(N \log N)$.  In our discussion, we focus on the general case (referred as \emph{dense} case) and sparse case.  For subspace decomposition type algorithms, including RCD, BRCD, and RFASD, each iteration only needs to compute $\nabla f(x^k)$ restricted on one subspace and, therefore, the cost of computing the gradient is $\mathcal{O}(N n_i)$ (dense case) or $\mathcal{O}(n_i)$ (sparse case).  Note $n_i = 1$ for RCD.  When the preconditioning technique is applied, the extra cost is introduced besides computing the gradient.  We assume computing the inverse of an $n \times n$ matrix is $\mathcal{O}(n^p)$ with $p\geq 1$, then the extra cost for PGD is $\mathcal{O}(N^p)$ since the need of $A^{-1}$.  For the proposed RFASD, the extra cost is reduced to $\mathcal{O}(n_i^p)$ since we only need to compute $A_i^{-1}$ on each subspace.  Now, we summarize the complexity comparison in Table~\ref{table:examples}.



\begin{table}[htp]
	\centering
	\caption{Complexity comparisons. ($\epsilon$: target accuracy; $L$: Lipschitz constant in $l^2$-norm; $\mu$: strong convexity constant in $l^2$-norm; $L_{_A}$: Lipschitz constant in $A$-norm; $\mu_{_A}$: strong convexity constant in $A$-norm; $\bar L := \frac{1}{J} \sum_{i=1}^J L_i$: averaged Lipschitz constant in $l^2$-norm; $\bar L_{_A} := \frac{1}{J} \sum_{i=1}^J L_{A,i}$: averaged Lipschitz constant in $A$-norm. GD: gradient descent; PGD: prreconditioned gradient descent; RCD: randomized coordinate descent; RBCD: randomized block coordinate descent; RFASD: randomized fast subspace descent.)}
	\renewcommand{\arraystretch}{2}
	\begin{tabular}{@{} c c c c c @{}}
	\toprule
		&   Convex  & Strongly convex &  \multicolumn{2}{c}{Cost per iteration} \\
		\cline{4-5}
		& & & dense & sparse 		
 \\  \hline
		GD & $\displaystyle \mathcal{O}\left(  \frac{L}{\epsilon} \right)$ & $\displaystyle \mathcal{O} \left(  \frac{L}{\mu} |\log \epsilon|\right)$ & $\displaystyle \mathcal{O}(N^2)$ & $\mathcal{O}(N)$ 

\medskip \\
		
		PGD & $\displaystyle \mathcal{O}\left(  \frac{L_{_A}}{\epsilon} \right)$ & $\displaystyle \mathcal{O} \left(  \frac{L_{_A}}{\mu_{_A}} |\log \epsilon|\right)$ & $\displaystyle \mathcal{O}(N^2 + N^p)$ & $\mathcal{O}(N^p)$
		\medskip \\   
		RCD & $\displaystyle \mathcal{O} \left( N \frac{\bar{L}}{\epsilon} \right) $ & $\displaystyle \mathcal{O}\left( N \frac{\bar{L}}{\mu} |\log \epsilon| \right)$& $\displaystyle \mathcal{O}(N)$ & $\mathcal{O}(1)$
		\medskip \\           
		RBCD & $\displaystyle \mathcal{O}\left( J \frac{\bar L}{\epsilon} \right)$ & $\displaystyle \mathcal{O}\left( J \frac{ \bar L }{\mu} |\log \epsilon| \right)$& $\displaystyle \mathcal{O}(N n_i)$ & $\mathcal{O}(n_i)$
	\medskip	\\           
		RFASD & $ \displaystyle \mathcal{O}\left( C_{_A} J \frac{\bar L_{_A}} {\epsilon} \right)$ & 
		$ \displaystyle \mathcal{O}\left( C_{_A} J  \frac{\bar L_{_A}}{\mu_{_{\small A}}} |\log \epsilon| \right)$ & $\displaystyle \mathcal{O}(N n_i+ n_i^p)$ & $\mathcal{O}(n_i^p)$ \medskip \\    
\bottomrule
	\end{tabular}
	\label{table:examples}
\end{table}

From Table~\ref{table:examples}, it is clear that RFASD can take advantage of the preconditioning effect, i.e., $\frac{L_{_A}}{\mu_A} \ll \frac{L}{\mu}$ or $\frac{\bar L_{_A}}{\mu_A} \ll \frac{\bar L}{\mu}$.  Meanwhile, there is no need to invert $A$ globally and but to compute $A_i^{-1}$ on each subspace, which reduces the computational cost in the sense that $\mathcal{O}(Nn_i + n_i^p) \ll \mathcal{O}(N^2 + N^p)$ (dense case) or $\mathcal{O}(n_i^p) \ll \mathcal{O}(N^p)$ (sparse case) if $n_i \ll N$ and $p > 1$.  Of course, the key is a stable space decomposition in $A$-norm such that the stability constant $C_{_A}$ can be kept small.  In next section, we use Nesterov's ``worst'' problem~\cite{nesterov2013introductory} as an example to demonstrate how to achieve this in practice.

\section{Examples} \label{sec:examples}
In this section, we give some examples of the RFASD method and use the example introduced by Nesterov~\cite{nesterov2013introductory} to discuss different methods. 

We first recall the Nesterov's ``worst'' problem~\cite{nesterov2013introductory}
\begin{example}\label{prob:NW}\rm
For $x \in \mathbb{R}^N$, consider the non-constrained minimization problem~\eqref{eq:min} with
\begin{equation}\label{eg:Nes_worst_fun}
f(x) := f_{L,r} (x) = \frac{L}{4} \left( \frac{1}{2} \left( x_1^2 + \sum_{i=1}^{r-1} (x_i - x_{i-1})^2 + x_r^2 \right) - x_1 \right),
\end{equation}
where $x_i$ represents the $i$-th coordinate of $x$ and $r < N$ is a constant integer that defines the intrinsic dimension of the problem. The minimum value of the function is 
\begin{equation}\label{soln:NW}
f_* = \frac{L}{16} \left( -1 + \frac{1}{r+1} \right). 
\end{equation}
\end{example}

\subsection{Randomized block coordinate descent methods} \label{sec:RBCD}
We follow~\cite{nesterov2012efficiency} to present the RBCD methods.  Let $\mathcal{V} = \mathbb{R}^N$ endowed with standard $\ell^2$-norm $\|\cdot\|$, i.e. $A = I$. Define a partition of the unit matrix 
$$
I_n = (U_1, \cdots, U_J) \in \mathbb{R}^{N \times N}, ~ U_i \in \mathbb{R}^{N \times n_i}, ~ i = 1,\cdots, J. 
$$
Now we consider the space decomposition $\mathcal{V} = \oplus_{i=1}^J \mathcal{V}_i,$ where $\mathcal{V}_i = \operatorname{Range}(U_i)$ and $\sum_{i=1}^J n_i = N$.  Naturally, $I_i = U_i$ and $R_i: U_i^{\intercal}$ in this setting. For each subspace, we also use the $\ell^2$-norm $\|\cdot\|$, i.e. $A_i = I_{n_i}$ is the identity matrix of size $n_i$. In this setting, RFASD (Algorithm~\ref{alg:RFASD}) is given by  
\begin{equation} \label{eqn:RCD}
x^{k+1} = x^k + \frac{1}{L_{i_k}} s_{i_k}, \quad s_{i_k} = - U_{i_k} U_{i_k}^{\intercal} \nabla f(x^k).
\end{equation}
This is just the RBCD algorithm proposed in~\cite{nesterov2012efficiency}.  Moreover, if the space decomposition is coordinate-wise, i.e., $n_i = 1$, $i=1, 2, \cdots, J = N$, then it is reduced to the RCD method. 

Regarding the convergence analysis, since the subspace decomposition is direct and orthogonal in $\ell^2$ inner product, in this case, we have
\begin{equation*}
\sum_{i=1}^J s_i = - \nabla f(x^k) \quad \text{and} \quad
\sum_{i=1}^J \| s_i \|^2 = \| \nabla f(x^k) \|^2,
\end{equation*}
which implies that $C_{_A} = 1$ in (SD). Moreover, because the $l^2$-norm is used here, Lipschitz constant and strong convexity constant are measured in $l^2$-norm and, hence, we drop the subscript ${_A}$ for those constants. Finally, we apply  Theorem~\ref{thm:str-con-linear} and \ref{thm:conv-sublinear} and recovery the classical convergence results of RBCD~\cite{nesterov2012efficiency} as follows,
\begin{itemize} 
	\item Convex case: $\displaystyle   \frac{2 J \bar L R_0^2}{k} $
	\item Strongly convex case: $\displaystyle  \left( 1 - \frac{\mu}{J \bar L}  \right)^k $
\end{itemize}
 


Consider Example~\ref{prob:NW} with $r = N$, since $l^2$-norm is used, it is easy to see that $\bar L \leq L$ and $\mu = \mathcal{O}(L\,N^{-2})$. Therefore, the condition number is $\frac{\bar L}{\mu} = \mathcal{O}(N^{2})$ and, for strongly convex case, the convergence rate is $\left( 1 - \frac{1}{J} \frac{1}{N^2} \right)$ and, according to Table~\ref{table:examples}, it requires $\mathcal{O}(N^3 |\log \epsilon|)$ operations (due to the fact that this problem is sparse) to achieve a given accuracy $\epsilon$.  This could be quite expensive, even impractical, for large $N$, i.e., large-scale problems. 

\subsection{Randomized fast subspace descent methods} \label{sec:RPSD}
The RFASD method allows us to use a preconditioner $A$ without computing $A^{-1}$. We chose an appropriate $A$-norm $\|\cdot\|_A$ defined by an SPD matrix $A$.  
%
Let 
\begin{equation*}
\mathcal{V} = \mathcal{V}_1 + \mathcal{V}_2 + \cdots + \mathcal{V}_J, \quad \mathcal{V}_i \subset \mathcal{V}, \quad i = 1,\cdots, J,
\end{equation*}
be a space decomposition. For each subspace, we still use the $A$-norm. Namely we use $\| v_i \|_{A_i} = \| v_i \|_{A}$ for all $v_i\in \mathcal V_i$. One can easily verify that $A_i = R_i AI_i$ which is the so-called Galerkin projection of $A$ to the subspace $\mathcal V_i$.  

%

In this case, the averaged Lipschitz constant $\bar L_{_A}$ and the strong convexity constant $\mu_{_A}$ are measured in $A$-norm. Based on Theorem~\ref{thm:str-con-linear} and \ref{thm:conv-sublinear}, we naturally have
\smallskip 
\begin{itemize} 
	\item Convex case: $\displaystyle \frac{2 J C_{_A} \bar L_{_A} R_0^2}{k} $
	\item Strongly convex case: $\displaystyle \left( 1 - \frac{1}{JC_{_A}}\frac{\mu_{_A}}{\bar L_{_A}} \right)^k  $
\end{itemize}
\smallskip
Comparing with the convergence results of RBCD, the key here is to design an appropriate preconditioner $A$ which induces the $A$-norm and corresponding stable decomposition such that $C_{_A} \bar L_{_A} \ll \bar L $ for convex case or $C_{_A}\bar L_{_A}/\mu_{_A} \ll \bar L/\mu$ for strongly convex case.  Then we will achieve speedup comparing with RCD/RBCD.

Of course, the choices of the preconditioner and the stable decomposition are usually problem-dependent.  Let us again consider Example~\ref{prob:NW} with $r = N$. Note that the objective function $f$ can be written in the following matrix format
\begin{align}\label{eqn:NW-matrix}
f(x) &= f_{L,N}(x) = \frac{1}{2} ( A x,x) - (x,  b), \\
 A &= \frac{L}{2} \operatorname{tridiag}(-1,2,-1) \in \mathbb{R}^{N \times N} \ \text{and} \ b = \frac{L}{4}e_1,
\end{align}
where $e_1 = (1, 0, \cdots, 0)^T \in \mathbb{R}^N$.  A good choice of the preconditioner is $A$ itself.  It is easy to verify that, when measuring in $A$-norm, the averaged Lipschitz constant is $\bar L_{_A} = 1$ and the strong convexity constant is $\mu_{_A} = 1$. Therefore, the condition number measured in $A$-norm is $\frac{\bar L_{_A}}{\mu_{_A}} = 1 \ll \mathcal{O}(N^2)$, i.e., much smaller than the condition number measure in $l^2$-norm.  

To achieve good overall performance, we also need to find a stable subspace decomposition to avoid inverting $A$ directly while keep $C_A$ small.  Note that $A$ (to be precise, $\frac{2}{L \, N^2} A$) here is essentially a central finite difference approximation of $u''(x)$ on the interval $[0,1]$ using a uniform mesh with $N$ subintervals.   A stable decomposition can be given by a multilevel decomposition used in the geometric multigrid method~\cite{trottenbergMultigrid2000,hackbuschMultiGridMethodsApplications2013}.  We postpone the details of this stable decomposition in Section~\ref{sec:numerics}.  Using such a decomposition, we have $n_i = 1$, $J = \mathcal{O}(N \log \, N)$, and $C_{_A} = \mathcal{O}(1)$. Therefore, RFASD has complexity $\mathcal{O}(N \log \, N |\log \, \epsilon|) $ to achieve a given accuracy $\epsilon$, which is quasi-optimal and scalable for large $N$. Comparing with the RCD/RBCD discussed in the previous section, the improvement is evident. 




\section{Numerical Experiments} \label{sec:numerics}

In this section, we present some numerical results for solving the Nesterov's worst function described in Example \ref{eg:Nes_worst_fun}.  We focus on the strongly convex case, i.e., $r=N$, to better demonstrate the advantages of using RFASD. The comparisons are made between RCD and RFASD.  The results based on RCD and RFASD with permutation (non-replacement sampling) (denoted by  RCD$_{\rm{perm}}$ and RFASD$_{\rm{perm}}$, respectively), cyclic CD, and cyclic FASD will also be presented to illustrate the influence of randomization.

For the implementation RCD and RCD$_{\rm{perm}}$, we follow~\cite{nesterov2012efficiency} and uses the step size $\alpha_k = \frac{1}{L_{i_k}}$, see~\eqref{eqn:RCD}.  For the Nestrov's worst function, i.e, Example \ref{eg:Nes_worst_fun}, we have $L_{i} = \| a_i \|$, where $a_i$ is the $i$-th column of $A = \frac{L}{2} \operatorname{tridiag}(-1,2,-1) \in \mathbb{R}^{N \times N}$, and, therefore, $\alpha_k = \frac{1}{\|a_i\|}$ is used in our numerical experiments.  The same step size is used for cyclic CD.  In addition, due to the definition of $A$, $L_1=L_N \approx L_2 = \cdots = L_{N-1}$. Thus, we use uniform sampling in RCD, i.e., $p_i = \frac{1}{N}$. 

According to the matrix format of the Nesterov's worst function~\eqref{eqn:NW-matrix}, we use $A = \operatorname{tridiag}(-1,2,-1) \in \mathbb{R}^{N\times N}$ as the preconditioner.  As suggested in~Section~\ref{sec:RPSD}, the stable decomposition used here is based on the geomerical multigrid methods~\cite{trottenbergMultigrid2000,hackbuschMultiGridMethodsApplications2013}. We refer to \cite{Yavneh:2006multigrid} for a brief introduction of the multigrid method. More precisely, we treat the $i$-th component, $x_i$,  as the function value at an artificial grid point at $ih$ with $h = 1/N$. We chose $N=2^{\rm{level}}-1$ and use a multilevel nodal decomposition as the space decomposition~\eqref{eqn:space-decomp}. Figure~\ref{fig:multilevel-decomp} illustrates such a multilevel space decomposition for the case $\mathcal{V} = \mathbb{R}^7$, i.e. $\rm{level}=3$. Note that we have $J < 2N$ and $n_i=1$, $i=1,\cdots, J$. With the choice of $A_i = R_iA I_i$, it is well-known that $C_A = \mathcal{O}(1)$ (see e.g. \cite{Xu:1992Iterative}).  Note that, in this setting, $L_{A, i} = 1$, $i=1,\cdots, J$, therefore, the step size is simply $\alpha_k = 1$ in our numerical experiments and uniform sampling, $p_i = \frac{1}{J}$, is used for RFASD.  For RFASD$_{\rm{perm}}$ and cyclic FASD, we use the same space decomposition and step size.  

\begin{figure}
\caption{Example: multilevel nodal decomposition for $\mathcal{V} = \mathbb{R}^7 = \mathcal{V}_1 + \cdots + \mathcal{V}_{11}$} \label{fig:multilevel-decomp}
\smallskip
\begin{center}
\begin{tikzpicture}[scale = 0.9]
\draw (0,0) -- (14,0);  
\filldraw [black] (1,0) circle (2pt);
\filldraw [black] (3,0) circle (2pt);
\filldraw [black] (5,0) circle (2pt);
\filldraw [black] (7,0) circle (2pt);
\filldraw [black] (9,0) circle (2pt);
\filldraw [black] (11,0) circle (2pt);
\filldraw [black] (13,0) circle (2pt);
\node[above] at (1,0) {\tiny $\mathcal{V}_1 = \operatorname{span}\{ e_1 \}$};
\node[above] at (3,0) {\tiny $\mathcal{V}_2 = \operatorname{span}\{ e_2 \}$};
\node[above] at (5,0) {\tiny $\mathcal{V}_3 = \operatorname{span}\{ e_3 \}$};
\node[above] at (7,0) {\tiny $\mathcal{V}_4 = \operatorname{span}\{ e_4 \}$};
\node[above] at (9,0) {\tiny $\mathcal{V}_5 = \operatorname{span}\{ e_5 \}$};
\node[above] at (11,0) {\tiny $\mathcal{V}_6 = \operatorname{span}\{ e_6 \}$};
\node[above] at (13,0) {\tiny $\mathcal{V}_7 = \operatorname{span}\{ e_7 \}$};
\draw (0,-2) -- (14,-2);  
\filldraw [black] (3,-2) circle (2pt);
\filldraw [black] (7,-2) circle (2pt);
\filldraw [black] (11,-2) circle (2pt);
\draw [->] (3,-1.5) -- (3, -0.2);
\draw [->] (2.9,-1.5) -- (1.1, -0.2);
\draw [->] (3.1,-1.5) -- (4.9, -0.2);
\node[above] at (3,-2) {\tiny $\mathcal{V}_8 = \operatorname{span}\{ \frac{1}{2}e_1 + e_2 + \frac{1}{2} e_3 \}$};
\node at (3,-0.8) {\tiny $1$};
\node at (2,-0.8) {\tiny $\frac{1}{2}$};
\node at (4,-0.8) {\tiny $\frac{1}{2}$};
\draw [->] (7,-1.5) -- (7, -0.2);
\draw [->] (6.9,-1.5) -- (5.1, -0.2);
\draw [->] (7.1,-1.5) -- (8.9, -0.2);
\node[above] at (7,-2) {\tiny $\mathcal{V}_9 = \operatorname{span}\{ \frac{1}{2}e_3 + e_4 + \frac{1}{2} e_5 \}$};
\node at (7,-0.8) {\tiny $1$};
\node at (6,-0.8) {\tiny $\frac{1}{2}$};
\node at (8,-0.8) {\tiny $\frac{1}{2}$};
\draw [->] (11,-1.5) -- (11, -0.2);
\draw [->] (10.9,-1.5) -- (9.1, -0.2);
\draw [->] (11.1,-1.5) -- (12.9, -0.2);
\node[above] at (11,-2) {\tiny $\mathcal{V}_{10} = \operatorname{span}\{ \frac{1}{2}e_5 + e_6 + \frac{1}{2} e_7 \}$};
\node at (11,-0.8) {\tiny $1$};
\node at (10,-0.8) {\tiny $\frac{1}{2}$};
\node at (12,-0.8) {\tiny $\frac{1}{2}$};
%
\draw (0,-4) -- (14,-4);  
\filldraw [black] (7,-4) circle (2pt);
\draw [->] (7,-3.5) -- (7, -2.2);
\draw [->] (6.9,-3.5) -- (3.1, -2.2);
\draw [->] (7.1,-3.5) -- (10.9, -2.2);
\node[above] at (7,-4) {\tiny $\mathcal{V}_{11} = \operatorname{span}\{ \frac{1}{4}e_1 + \frac{1}{2} e_2 + \frac{3}{4} e_3 + e_4 + \frac{3}{4} e_5 + \frac{1}{2} e_6 + \frac{1}{4} e_7 \}$};
\node at (7,-2.8) {\tiny $1$};
\node at (5,-2.8) {\tiny $\frac{1}{2}$};
\node at (9,-2.8) {\tiny $\frac{1}{2}$};
\end{tikzpicture}
\end{center}
\end{figure}

For all the experiments, we choose initial guess to be the vector $(1,\cdots,1)^T \in \mathbb{R}^N$ and stop the iterations when the relative norm of the gradient satisfies $\frac{\| \nabla f(x^k) \|}{\| \nabla f(x^0) \|} \leq 10^{-6}$.  Due to the randomness in the algorithms, we repeat each test $10$ times and report the average results.

\begin{table}[h]
\centering
\caption{Performance of RCD, RCD$_{\rm{perm}}$, and cyclic RCD (for problems of size $N\geq 127$, all three methods take more than $10^6$ iterations to converges and, therefore, not reported here.)} \label{tab:CD}
\begin{tabular}{@{} c c c c c c c @{}}
\toprule
Size & \multicolumn{2}{c}{RCD} & \multicolumn{2}{c}{RCD$_{\rm{perm}}$} & \multicolumn{2}{c}{cyclic CD}	\\ \cline{2-7}
$N$ & $\#$Iter. & $\#$Epoch & $\#$Iter. & $\#$Epoch & $\#$Iter. & $\#$Epoch \\ \hline 
7     &  1.4129e3  & 201.84   & 944.50   & 134.93   & 819      & 117 \\ 
15    &  1.1054e4  & 736.96   & 7.4658e3 & 497.72   & 6.465e3  & 431 \\
31    &  8.3177e4  & 2.6831e3 & 5.6284e4 & 1.8156e3 & 4.8576e4 & 1.5670e3 \\
63    &  6.1011e5  & 9.6843e3 & 4.1218e5 & 6.5425e3 & 3.5519e5 & 5.6379e3 \\
\bottomrule
\end{tabular}
\end{table}

\begin{table}[h]
\centering
\caption{Performance of RFASD, RFASD$_{\rm{perm}}$, cyclic FASD} \label{tab:FASD}
\begin{tabular}{@{} c c c c c c c c @{}}
\toprule
\multicolumn{2}{c}{size} & \multicolumn{2}{c}{RFASD} & \multicolumn{2}{c}{RFASD$_{\rm{perm}}$} & \multicolumn{2}{c}{cyclic FASD}	\\ \hline
$N$ & $J$ & $\#$Iter. & $\#$Epoch & $\#$Iter. & $\#$Epoch & $\#$Iter. & $\#$Epoch \\ \hline 
7     & 11    & 104.90   & 9.54  & 46.90    & 4.26 & 69       & 6.27 \\
15    & 26    & 323.20   & 12.43 & 157.30   & 6.05 & 213      & 8.19 \\
31    & 57    & 830.70   & 14.57 & 467.20   & 8.20 & 518      & 9.09 \\ 
63    & 120   & 1.9249e3 & 16.04 & 975.40   & 8.13 & 1.103e3  & 9.19 \\
127   & 247   & 3.8384e3 & 15.54 & 2.0484e3 & 8,29 & 2.278e3  & 9.22 \\
255   & 502   & 7.7405e3 & 15.42 & 4.301e3  & 8.57 & 4.637e3  & 9.24 \\
511   & 1,013 & 1.5797e4 & 15.59 & 8.2902e3 & 8.18 & 9.632e3  & 9.51 \\
1,023 & 2,036 & 3.2071e4 & 15.75 & 1.7563e4 & 8.63 & 1.9352e4 & 9.50 \\
2,047 & 4,083 & 6.7138e4 & 16.44 & 3.3591e4 & 8.23 & 3.88e4   & 9.50 \\
4,095 & 8,178 & 1.2813e5 & 15.67 & 6.9891e4 & 8.55 & 7.7701e4 & 9.50 \\
\bottomrule
\end{tabular} 
\end{table}

In Table~\ref{tab:CD} and~\ref{tab:FASD}, we compare all the algorithms by reporting the number of iterations ($\#$Iter.) and the number of epochs ($\#$Epoch). Here, we borrow the terminology from machine leaning and refer all the subspaces as one epoch.  Therefore, the number of epochs is defined as $ \displaystyle \# \text{Epoch} = \frac{\#\text{Iter.}}{J}$.  From the two tables, we can immediately see that the FASD-type algorithms outperform CD-type algorithms. For CD-type algorithms, we can see that the number of iterations grows like $\mathcal{O}(N^3)$, which confirms our discussion in Section~\ref{sec:RBCD}.  Since $J=N$ in this case, the number of epochs behaves like $\mathcal{O}(N^2)$, which is consistent with what we see from Table~\ref{tab:CD}. Among the three CD-type algorithms, RCD$_{\rm{perm}}$ uses the least number of iterations while RCD uses the most.  For FASD-type algorithms, we report the number of subspaces obtained by the multilevel decomposition. Due to such a construction,  we can see that $J < 2N$. Following the discussion in Section~\ref{sec:RPSD}, we expect that the number of iterations and the number of epochs behave like $\mathcal{O}(N)$ and $\mathcal{O}(1)$, respectively.  This is confirmed by the numerical results in Table~\ref{tab:FASD} and demonstrates that, due to the good choice of the preconditioner and proper subspace decomposition, RFASD and its variants, RFASD$_{\rm{perm}}$ and cyclic FASD, can achieve optimal computational complexity for Example~\ref{prob:NW}. Finally, for FASD-type algorithms, we observe that RFASD$_{\rm{perm}}$ seems to be the best choice.  To better understand the convergence of RFASD$_{\rm{perm}}$, especially its comparison with RFASD is an open question and a subject of our future research.  Nevertheless, we recommend to use RFASD$_{\rm{perm}}$ in practice if possible.

\section{Conclusions and Future Work}\label{sec:conclusions}
In this paper, we have derived a randomized version of fast subspace descent methods. We first find a stable space decomposition for $\mathbb R^N$ and then randomly chose a subspace to apply a gradient descent step.  

We have also developed the convergence analysis which shows the convergence of the proposed method can be measured by the condition number in a new $A$-norm and thus can achieve faster convergence if the condition number is much smaller than the standard one measured in $\ell^2$-norm and the space decomposition is $A$-stable. 

For the Nesterov's ``worst" problem, we have found a stable and multilevel space decomposition and shown both theoretically and numerically the optimal convergence rate. However, we haven't discussed how to design a stable space decomposition for other benchmark optimization problems, which is a subject of our ongoing work.

In the application of data science, there is no natural grid hierarchy to construct the space decomposition. 
Algebraic multigrid methods (AMG) will be a more appropriate approach.
There is still very little theory and algorithm on AMG for nonlinear equations and optimization problems~\cite{Treister;Turek;Yavneh:2016multilevel,Ponce;Bindel;Vassilevski:2018Nonlinear}. In the future, we propose to develop an algebraic FASD method for nonlinear optimization problems. 


\bibliographystyle{abbrv}
\bibliography{rfas}
	
\end{document}